\documentclass[11pt]{amsart}

\usepackage[utf8]{inputenc}
\usepackage[english]{babel}
\usepackage{amsmath,amssymb,amsthm,mathtools}
\usepackage{graphicx}
\usepackage{float}
\usepackage{xcolor}
\usepackage[colorlinks=true,linkcolor=blue,urlcolor=blue,citecolor=blue]{hyperref}

\graphicspath{{images/}}

\newtheorem{theorem}{Theorem}[section]
\newtheorem{lemma}[theorem]{Lemma}
\newtheorem{proposition}[theorem]{Proposition}

\theoremstyle{definition}

\theoremstyle{remark}
\newtheorem{remark}[theorem]{Remark}

\newcommand{\M}[1]{\mathcal{M}(#1)}
\newcommand{\ov}[1]{\overline{#1}}

\begin{document}

\title{Small generating sets for handlebody groups}

\author{T{\"u}l\.{i}n Altun{\"o}z}
\address{Faculty of Engineering, Ba\c{s}kent University, Ankara, Turkey}
\email{tulinaltunoz@baskent.edu.tr}

\author{Celal Can Bellek}
\address{Department of Mathematics, University of Wisconsin-Madison, Madison, United States}
\email{bellek@wisc.edu}

\author{Em\.{i}r G{\"u}l}
\address{Department of Mathematics, Middle East Technical University, Ankara, Turkey}
\email{gul.emir@metu.edu.tr}

\author{Mehmetc\.{i}k Pamuk}
\address{Department of Mathematics, Middle East Technical University, Ankara, Turkey}
\email{mpamuk@metu.edu.tr}

\author{O\u{g}uz Y{\i}ld{\i}z}
\address{Department of Mathematics, Middle East Technical University, Ankara, Turkey}
\email{oguzyildiz16@gmail.com}

\thanks{This work is supported by the Scientific and Technological Research Council of Turkey (T\"UB\.ITAK), grant number 125F253.}

\begin{abstract}
Wajnryb gave a finite presentation of the handlebody group $\M{V_g}$ and showed that it is generated by five elements for $g\geq 2$. Using his relations, we show that $\M{V_g}$ is generated by three elements if $g\geq 2$.
\end{abstract}

\maketitle

\section{Introduction}

Let $V_g$ be a $3$-dimensional handlebody of genus $g$. Its mapping class group $\M{V_g}$ is called the handlebody group. Restricting to the boundary identifies $\M{V_g}$ with a subgroup of the mapping class group of the closed surface $\Sigma_g=\partial V_g$, and most of what is known about handlebody groups is proved through this identification.

Suzuki~\cite{Suzuki} gave the first finite generating set for $\M{V_g}$. It consists of six homeomorphisms described geometrically. Takahashi~\cite{Takahashi} wrote these generators as products of Dehn twists on $\Sigma_g$. Wajnryb~\cite{Wajnryb1998} found a finite presentation of $\M{V_g}$ for $g\geq 2$ and, using it, a generating set with five elements. For the mapping class group of a closed surface the corresponding question is settled: it is generated by two elements~\cite{Wajnryb1996, Korkmaz2005}, and two is optimal because the group is not cyclic. For the balanced superelliptic handlebody group, a subgroup of $\M{V_g}$, Omori~\cite{Omori2023} found a small generating set. In this paper we reduce the number of generators of $\M{V_g}$ itself.

For a group $G$ we write $d(G)$ for the smallest number of elements that generate $G$. The elements $R$, $t_1$, $A_i$, $s_i$ and $r_{1,2}$ appearing below are defined in Section~\ref{sec:prelim}; $R$ is a rotation of $V_g$ of order $g$, and the others are products of Dehn twists taken from~\cite{Wajnryb1998}.

\begin{theorem}\label{thm:three}
For $g\geq 3$, the handlebody group $\M{V_g}$ is generated by the three elements
\[
R,\qquad t_1,\qquad r_{1,2}\,s_3A_2 .
\]
In particular, $d(\M{V_g})\leq 3$ for $g\geq 3$.
\end{theorem}
\begin{theorem}\label{thm:threefortwo}
The handlebody group $\M{V_2}$ is generated by the three elements
\[
t_1,\qquad s_1,\qquad \,r_{1,2}A_2 .
\]
In particular, $d(\M{V_2})\leq 3$.
\end{theorem}

For $g=2$ the product $u=t_1\cdots t_{g-1}$ in Wajnryb's generating set (Theorem~\ref{thm:wajnryb}) is just $t_1$, so his set already has four elements and $d(\M{V_2})\leq 4$. On the other hand, $\M{V_g}$ is not abelian for $g\geq 2$, so $d(\M{V_g})\geq 2$ (Proposition~\ref{prop:lower}). Thus, $2\leq d(\M{V_g})\leq 3$ for $g \geq 2$. We do not know whether $\M{V_g}$ is generated by two elements.

In both proofs we show that the subgroup generated by our elements contains Wajnryb's five generators (for $g=2$, four generators). The rotation $R$ moves an element supported near one handle to the corresponding element at another handle, and a braid relation between $t_1$ and $r_{1,2}$ lets us split $r_{1,2}$ off from the product generator. For Theorem~\ref{thm:three} we also need the lantern relation and a relation of Wajnryb that conjugates a twist about a curve enclosing two spots to a twist about a curve enclosing three spots.

The paper is organised as follows. Section~\ref{sec:prelim} recalls Wajnryb's models, elements and the relations we need. Section~\ref{sec:proofs} contains the proof of Theorems~\ref{thm:three}. Section~\ref{sec:lower} discusses lower bounds.

\section{Preliminaries}\label{sec:prelim}

\subsection{Conventions}
Let $\Sigma_g$ be a closed orientable surface of genus $g$ and $\M{\Sigma_g}$ its mapping class group. For a simple closed curve, we denote the right-handed Dehn twist about it by the corresponding capital Latin letter; for instance $A_i$, $B_i$, $E_i$ and $C_I$ are the twists about $\alpha_i$, $\beta_i$, $\epsilon_i$ and $\gamma_I$. The inverse of a mapping class $f$ is written $\ov{f}$. Products are composed from right to left: $fg$ means that $g$ is applied first. For $f,h\in\M{\Sigma_g}$ we write
\[
f^{h}=h\,f\,\ov{h}.
\]
With these conventions, $(T_\delta)^{h}=T_{h(\delta)}$ for every simple closed curve $\delta$ and every $h \in \M{\Sigma_g}$ \cite[Fact~3.7]{FarbMargalit}. Dehn twists about disjoint curves commute.

\subsection{Handlebody groups}
Let $V_g$ be a handlebody of genus $g$ with $\partial V_g=\Sigma_g$. Restricting homeomorphisms to the boundary gives a homomorphism $\M{V_g}\to\M{\Sigma_g}$, which is injective \cite{HandlebodyPrimer}. We use it to regard $\M{V_g}$ as a subgroup of $\M{\Sigma_g}$; it consists of the mapping classes that extend over $V_g$.

An essential simple closed curve $\alpha\subset\Sigma_g$ is a \emph{meridian} if it bounds a properly embedded disk in $V_g$. A collection of $g$ disjoint meridians whose complement in $\Sigma_g$ is connected is a \emph{cut system}. By Griffiths~\cite{Griffiths1964Automorphisms, griffithssomeele}, a mapping class $\phi\in\M{\Sigma_g}$ lies in $\M{V_g}$ if and only if $\phi$ maps every meridian to a meridian, and this holds as soon as $\phi$ maps the curves of one cut system to meridians; see also \cite{HandlebodyPrimer}. In particular, the Dehn twist about a meridian lies in $\M{V_g}$.

\subsection{Wajnryb's models and elements}\label{subsec:wajnryb}

We use the two models of $V_g$ from~\cite{Wajnryb1998}. The first one, with the curves $\alpha_i$, $\beta_i$ and $\epsilon_i$, is shown in Figure~\ref{fig:wajmod}. The curves $\alpha_1,\dots,\alpha_g$ form a cut system. Cutting $V_g$ along the corresponding meridian disks gives a $3$-ball whose boundary sphere carries $2g$ spots $\alpha_{\pm1},\dots,\alpha_{\pm g}$, where the spots $\alpha_{-i}$ and $\alpha_i$ are glued back together to recover $V_g$. This is the second model, shown in Figure~\ref{fig:2ndmodel}. Since $\alpha_{-i}$ and $\alpha_i$ are the same curve of $\Sigma_g$, we have $A_{-i}=A_i$.

\begin{figure}[H]
\centering
\includegraphics[width=0.7\textwidth]{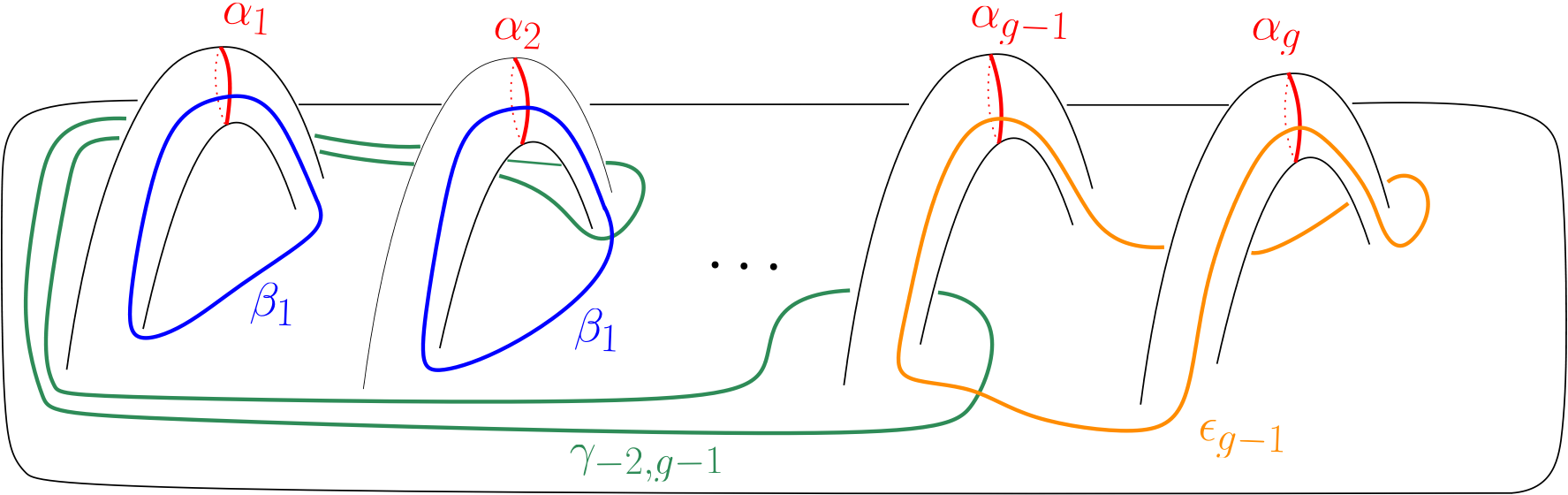}
\caption{The handlebody $V_g$ with the curves $\alpha_i$, $\beta_i$, $\gamma_{i,j}$ and $\epsilon_i$.}
\label{fig:wajmod}
\end{figure}

\begin{figure}[H]
\centering
\includegraphics[width=0.6\textwidth]{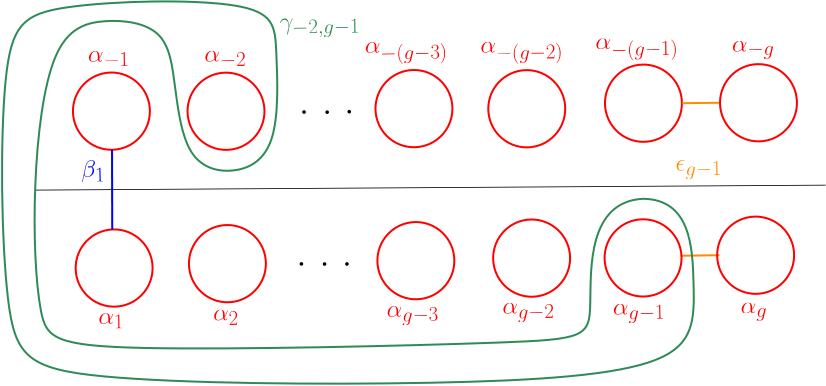}
\caption{The second model of $V_g$ with the curves $\alpha_i$, $\beta_1$, $\gamma_{-2,g-1}$ and $\epsilon_{g-1}$.}
\label{fig:2ndmodel}
\end{figure}

Let $I_0=\{-g,\dots,-1,1,\dots,g\}$. For $I\subset I_0$ we write $\gamma_I$ for the curve on the boundary sphere of the second model that separates the spots $\alpha_i$, $i\in I$, from the remaining spots, drawn as in~\cite{Wajnryb1998}; an example is shown in Figure~\ref{fig:gammaI}. For $I=\{i,j\}$ we write $\gamma_{i,j}$. Each $\gamma_I$ bounds a disk in the $3$-ball, so it is a meridian of $V_g$ and $C_I\in\M{V_g}$.

\begin{figure}[H]
\centering
\includegraphics[width=0.6\textwidth]{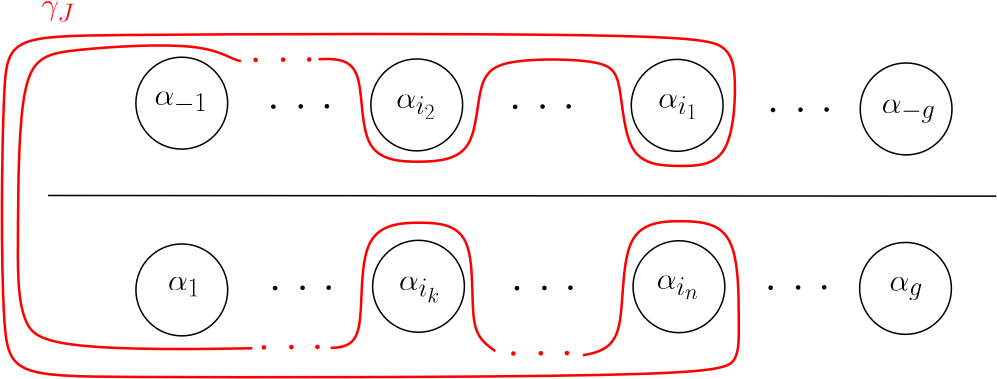}
\caption{A curve $\gamma_I$ enclosing the spots indexed by $I=\{i_1,\dots,i_n\}\subset I_0$.}
\label{fig:gammaI}
\end{figure}

Following Wajnryb~\cite{Wajnryb1998}, consider the elements
\[
t_i=E_iA_iA_{i+1}E_i\quad (1\leq i\leq g-1),\qquad
s_i=B_iA_i^2B_i\quad (1\leq i\leq g),\qquad
r_{1,2}=B_2A_2C_{1,2}B_2 .
\]
They all lie in $\M{V_g}$. The element $t_i$ interchanges the spots $\alpha_i$ and $\alpha_{i+1}$, and also $\alpha_{-i}$ and $\alpha_{-i-1}$, fixing the other spots (Figure~\ref{fig:t1}). The knob twist $s_i$ interchanges $\alpha_{-i}$ and $\alpha_i$ (Figure~\ref{fig:s}). The element $r_{1,2}$ interchanges the meridians $\alpha_2$ and $\gamma_{1,2}$ and fixes $\alpha_k$ for $k\neq 2$. The element $s_i$ is supported in a regular neighbourhood of $\alpha_i\cup\beta_i$, and $t_i$ in a regular neighbourhood of $\epsilon_i\cup\alpha_i\cup\alpha_{i+1}$.

\begin{figure}[H]
\centering
\includegraphics[width=0.45\textwidth]{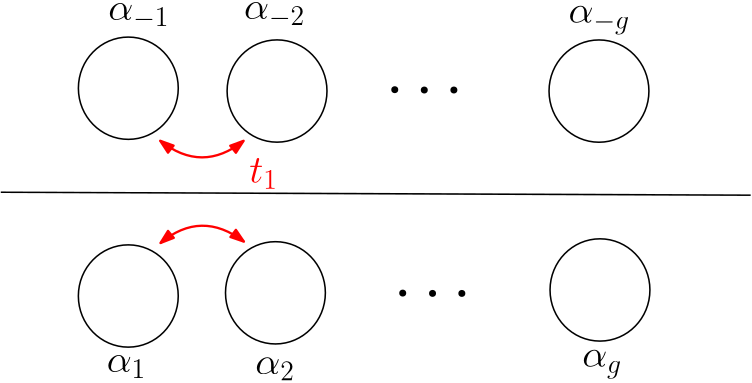}
\caption{The action of $t_1$ on the spots.}
\label{fig:t1}
\end{figure}

\begin{figure}[H]
\centering
\includegraphics[width=0.45\textwidth]{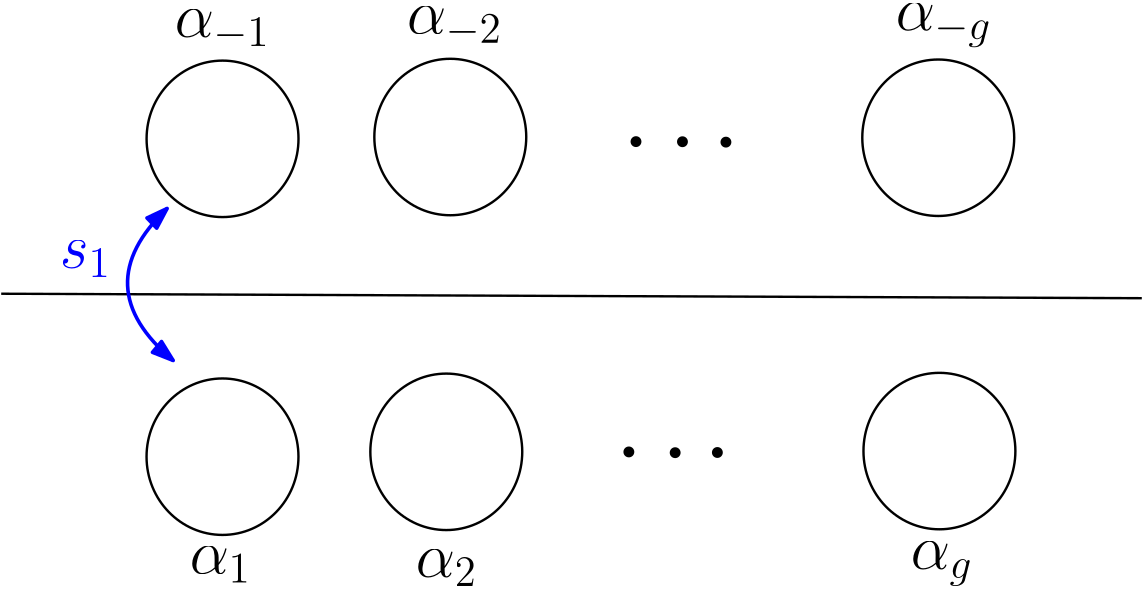}
\caption{The action of $s_1$ on the spots.}
\label{fig:s}
\end{figure}

\begin{theorem}[{Wajnryb~\cite{Wajnryb1998}}]\label{thm:wajnryb}
For $g\geq 2$, $\M{V_g}$ is generated by $A_1$, $s_1$, $r_{1,2}$, $t_1$ and $u=t_1t_2\cdots t_{g-1}$.
\end{theorem}

\subsection{Relations}\label{subsec:relations}
We list the relations used in the proofs. Relations (R2)--(R9) are among the defining relations of Wajnryb's presentation \cite[Theorem~1]{Wajnryb1998}, written in our notation, and (R10) is the case of three spots of his generalized lantern relation \cite{Wajnryb1998}. Relation (R1) holds because the curves involved are disjoint.

\begin{itemize}
\item[(R1)] Every $A_i$ commutes with every $A_j$, every $C_I$ and every $s_j$. The elements $t_1$ and $s_3$ commute, and for $g\geq 5$ so do $t_4$ and $s_3$.
\item[(R2)] $A_i^{t_i}=A_{i+1}$ for $1\leq i\leq g-1$, and $[A_k,t_i]=1$ for $k\notin\{i,i+1\}$.
\item[(R3)] $t_i^2=C_{i,i+1}\,C_{-i-1,-i}\,\ov{A_i}^{\,2}\,\ov{A_{i+1}}^{\,2}$ for $1\leq i\leq g-1$.
\item[(R4)] $[t_1,C_{1,2}]=1$.
\item[(R5)] For $1\leq i<j\leq 3$, the twist $C_{i,j}$ is a conjugate of $C_{1,2}$ by an element $w_{i,j}$ of the subgroup $\langle t_1,t_2\rangle$.
\item[(R6)] $A_2^{r_{1,2}}=C_{1,2}$; $[r_{1,2},A_k]=1$ for $k\neq 2$; and $[r_{1,2},s_k]=1$ for $k\geq 3$.
\item[(R7)] $t_1^{\,r_{1,2}}=r_{1,2}^{\,\ov{t_1}}$.
\item[(R8)] $r_{1,2}^2=s_2\,C_{1,2}\,s_2\,\ov{C_{1,2}}$.
\item[(R9)] If $g\geq 3$, then $C_{-3,-2}^{\,\ov{r_{1,2}}}=C_{\{1,2,3\}}^{\,\ov{s_3}}$.
\item[(R10)] $C_{\{1,2,3\}}=C_{1,2}\,C_{1,3}\,C_{2,3}\,\ov{A_1A_2A_3}$.
\end{itemize}

We record two consequences. First, by (R1) and (R3) the element $t_i^2$ commutes with $A_i$, so (R2) gives
\[
A_{i+1}^{t_i}=t_i^2A_i\ov{t_i}^{\,2}=A_i .
\]
Hence conjugation by the subgroup $\langle t_1,\dots,t_{g-1}\rangle$ permutes $\{A_1,\dots,A_g\}$, and every permutation of this set is realised. Second, writing (R7) as $r_{1,2}t_1\ov{r_{1,2}}=\ov{t_1}\,r_{1,2}\,t_1$ and multiplying by $t_1$ on the left and by $r_{1,2}$ on the right gives the braid relation
\begin{equation}\label{eq:braid}
t_1\,r_{1,2}\,t_1=r_{1,2}\,t_1\,r_{1,2}.
\end{equation}

\section{Proofs of the main theorems}\label{sec:proofs}

Let $R$ be the rotation of $V_g$ by $2\pi/g$ about the vertical axis through its centre, shown on the two models in Figures~\ref{fig:rotation1} and~\ref{fig:rotation2}. It is an orientation-preserving homeomorphism of $V_g$ and $R^g=1$.

\begin{figure}[H]
\centering
\begin{minipage}{0.335\textwidth}
\centering
\includegraphics[width=\linewidth]{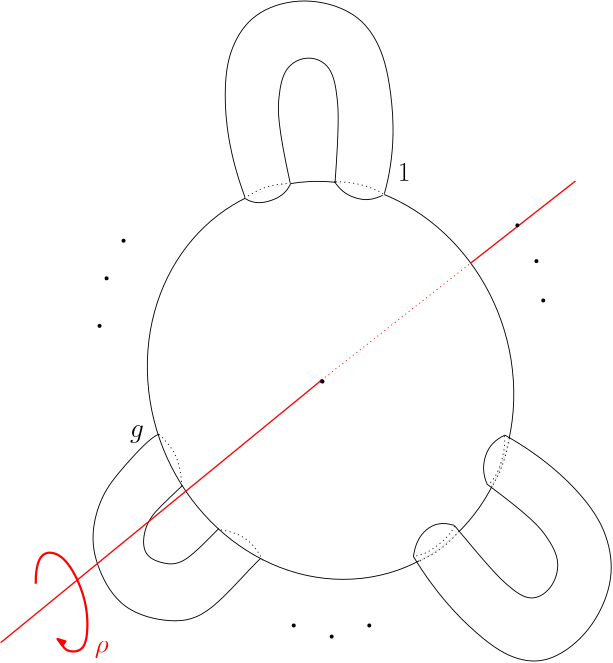}
\caption{The rotation $R$ on the first model.}
\label{fig:rotation1}
\end{minipage}
\hspace{0.2\textwidth}
\begin{minipage}{0.335\textwidth}
\centering
\includegraphics[width=\linewidth]{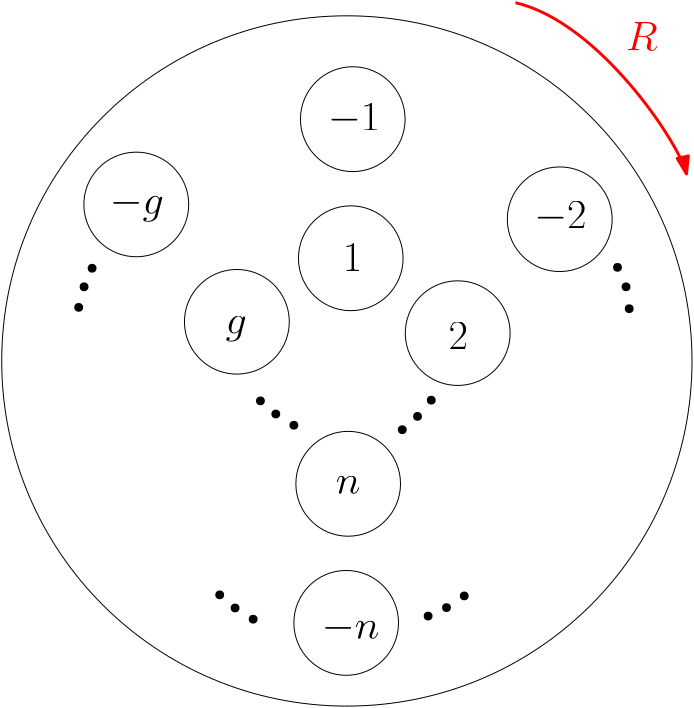}
\caption{The rotation $R$ on the second model.}
\label{fig:rotation2}
\end{minipage}
\end{figure}

\begin{lemma}\label{lem:rotation}
The rotation $R$ maps $\alpha_{\pm i}$ to $\alpha_{\pm(i+1)}$ and $\beta_i$ to $\beta_{i+1}$ for $1\leq i\leq g-1$, and $\epsilon_i$ to $\epsilon_{i+1}$ for $1\leq i\leq g-2$. Consequently
\[
A_i^{\,R}=A_{i+1},\qquad s_i^{\,R}=s_{i+1}\quad(1\leq i\leq g-1),\qquad t_i^{\,R}=t_{i+1}\quad(1\leq i\leq g-2),
\]
and conjugation by $\ov{R}$ lowers these indices by one.
\end{lemma}

\begin{proof}
In the second model the spots $\alpha_{\pm1},\dots,\alpha_{\pm g}$ and the curves $\beta_i$ and $\epsilon_i$ are placed symmetrically with respect to the rotation axis, and $R$ moves each of them to the next one, as stated. The formulas follow from $(T_\delta)^{R}=T_{R(\delta)}$ and the definitions of $s_i$ and $t_i$.
\end{proof}

The next lemma is the step that both proofs begin with.

\begin{lemma}\label{lem:split}
Let $G$ be a subgroup of $\M{V_g}$ containing $t_1$, and let $X\in\M{V_g}$ commute with both $t_1$ and $r_{1,2}$. If $r_{1,2}X\in G$, then $r_{1,2}\in G$ and $X\in G$.
\end{lemma}

\begin{proof}
Put $F=r_{1,2}X$. Since $X$ commutes with $t_1$ and $r_{1,2}$, the braid relation~\eqref{eq:braid} gives
\[
\ov{t_1}\,F\,t_1\,F\,\ov{t_1}
=X^2\,\ov{t_1}\,(r_{1,2}\,t_1\,r_{1,2})\,\ov{t_1}
=X^2\,\ov{t_1}\,(t_1\,r_{1,2}\,t_1)\,\ov{t_1}
=X^2 r_{1,2}.
\]
Hence $X^2r_{1,2}\in G$, and $X=(X^2r_{1,2})\,\ov{F}\in G$. Then $r_{1,2}=F\,\ov{X}\in G$.
\end{proof}

\begin{proof}[Proof of Theorem~\ref{thm:three}]
Let $g\geq 3$, and let $G=\langle R,\,t_1,\,r_{1,2}s_3A_2\rangle$ and set $F=r_{1,2}s_3A_2$. We show that $G$ contains $A_1$, $s_1$, $r_{1,2}$, $t_1$ and $u$. Throughout the proof, we use (R1) without comment.

\smallskip
\noindent\emph{Step 1: $r_{1,2},F,u\in G$.}
 The element $s_3$ commutes with $t_1$ and with $r_{1,2}$ by (R6); the twist $A_2$ has the equality
$r_{1,2}A_2 = C_{1,2}r_{1,2}$ by (R6) and $A_2 \overline{t_1} = \overline{t_1} A_1$ by (R2) and braid relation. Then, using the similar approach as in Proof of Lemma~\ref{lem:split}, we have
\[
    \overline{t_1} F t_1 F \overline{t_1} = \overline{t_1}r_{1,2}s_3 A_2 t_1 r_{1,2}s_3 A_2 \overline{t_1} = s_3^2 C_{1,2} \overline{t_1} r_{1,2} t_1 r_{1,2} \overline{t_1} A_1 = s_3^2 C_{1,2} r_{1,2} A_1 = s_3^2 C_{1,2} A_1 r_{1,2} \in G,
\]
where the last equality comes from (R6). Multiplicating this element by $\overline{F}$ yields
\[
    s_3^2 C_{1,2} A_1 r_{1,2} \overline{F} = s_3^2 C_{1,2} A_1 r_{1,2} \overline{A_2} \mkern3mu \overline{s_3} \mkern3mu \overline{r_{1,2}} = s_3^2 \overline{s_3} C_{1,2} \overline{C_{1,2}} A_1 = s_3 A_1 \in G.
\]
We know that $A_1^{t_1}=A_2$. Hence $X= (s_3A_1)^{t_1}=s_3A_2$. Using $X$, we get
\[
    F \overline{X} = r_{1,2}s_3 A_2 \overline{A_2} \mkern3mu \overline{s_3} = r_{1,2} \in G.
\]

\smallskip
\noindent\emph{Step 2: $A_i\ov{A_j}\in G$ for all $i,j$.}
By Step 1, we have $s_3A_1$ and $X = s_3 A_2$ are in $G$. Then,
\[
A_1\,\ov{A_2}=s_3 A_1 \,\ov{X}\in G .
\]
Conjugating by elements of $\langle t_1,\dots,t_{g-1}\rangle\subset G$ permutes the $A_i$ arbitrarily, so $A_i\ov{A_j}\in G$ for all $i$ and $j$.

\smallskip
\noindent\emph{Step 3: $s_2A_k\in G$ for all $k$.}
By Lemma~\ref{lem:rotation}, $X^{\ov{R}}=s_2A_1\in G$. By Step~2, $s_2A_k=(s_2A_1)(\ov{A_1}A_k)\in G$.

\smallskip
\noindent\emph{Step 4: $C_{1,2}\ov{A_k}$, $C_{1,3}\ov{A_2}$ and $C_{2,3}\ov{A_3}$ lie in $G$ for $k=1,2,3$.}
For $k\neq 2$, relation (R6) gives $(A_k\ov{A_2})^{r_{1,2}}=A_k\ov{C_{1,2}}$, so its inverse $C_{1,2}\ov{A_k}$ lies in $G$. Conjugating $C_{1,2}\ov{A_1}$ by $t_1$ and using (R2) and (R4), we get $C_{1,2}\ov{A_2}\in G$.

Now let $w=w_{1,3}$ be as in (R5). Conjugation by $w\in\langle t_1,t_2\rangle$ permutes $A_1,A_2,A_3$, say $A_k^{\,w}=A_{\sigma(k)}$, and $w$ moves the boundary of spots $\alpha_1,\alpha_2,\alpha_3$ according to $\sigma$. Since $w(\gamma_{1,2})=\gamma_{1,3}$ (here, $w$ is $t_2$ ) and $\gamma_{1,2}$ encloses exactly the boundary of spots $\alpha_1,\alpha_2$, we have $\sigma(\{1,2\})=\{1,3\}$, so $\sigma(3)=2$. Therefore $(C_{1,2}\ov{A_3})^{w}=C_{1,3}\ov{A_2}\in G$. In the same way, for $w'=w_{2,3}$ (here, $w'$ is $t_1 t_2$) the induced permutation $\sigma'$ satisfies $\sigma'(\{1,2\})=\{2,3\}$, so $\sigma'(k)=3$ for some $k\in\{1,2\}$, and $(C_{1,2}\ov{A_k})^{w'}=C_{2,3}\ov{A_3}\in G$.

\smallskip
\noindent\emph{Step 5: $C_{\{1,2,3\}}\in G$.}
The $A_i$ commute with the $C_I$, so (R10) can be rewritten as
\[
C_{\{1,2,3\}}=(C_{1,2}\ov{A_1})\,(C_{1,3}\ov{A_2})\,(C_{2,3}\ov{A_3}),
\]
and each factor lies in $G$ by Step~4.

\smallskip
\noindent\emph{Step 6: $A_k^2\in G$ for all $k$.}
By (R8), $\ov{s_2}\,r_{1,2}^2=C_{1,2}\,s_2\,\ov{C_{1,2}}$. Using Steps~3 and~4,
\[
(A_1\ov{C_{1,2}})\,\ov{(s_2A_1)}\;r_{1,2}^2
=A_1\ov{C_{1,2}}\,\ov{A_1}\,C_{1,2}\,s_2\,\ov{C_{1,2}}
=s_2\,\ov{C_{1,2}}\in G .
\]
Then $s_2\ov{A_1}=(s_2\ov{C_{1,2}})(C_{1,2}\ov{A_1})\in G$, and
\[
A_1^2=\ov{(s_2\ov{A_1})}\,(s_2A_1)\in G .
\]
Conjugating by elements of $\langle t_1,\dots,t_{g-1}\rangle$ gives $A_k^2\in G$ for every $k$.

\smallskip
\noindent\emph{Step 7: $C_{-3,-2}\ov{A_1}\in G$.}
The four factors on the right-hand side of (R3) for $i=2$ commute with each other, so
\[
C_{-3,-2}\,\ov{A_1}=(\ov{C_{2,3}}\,A_3)\;t_2^{\,2}\;A_2^{\,2}\;(A_3\ov{A_1}).
\]
Each factor lies in $G$, by Step~4, Step~1, Step~6 and Step~2 respectively.

\smallskip
\noindent\emph{Step 8: $C_{\{1,2,3\}}\ov{A_1}\in G$.}
By (R6) and (R9),
\[
\bigl(C_{-3,-2}\,\ov{A_1}\bigr)^{\ov{r_{1,2}}}=C_{\{1,2,3\}}^{\,\ov{s_3}}\;\ov{A_1}\in G .
\]
We conjugate this element by $X=s_3A_2\in G$. The curve $\alpha_2$ is disjoint from $\alpha_1$ and from $\ov{s_3}(\gamma_{\{1,2,3\}})$, because $s_3$ is supported near $\alpha_3\cup\beta_3$. Hence $A_2$ commutes with both factors, and conjugation by $X$ acts on them as conjugation by $s_3$. Since $s_3$ also commutes with $A_1$, we obtain
\[
\Bigl(C_{\{1,2,3\}}^{\,\ov{s_3}}\;\ov{A_1}\Bigr)^{X}=C_{\{1,2,3\}}\,\ov{A_1}\in G .
\]

\smallskip
\noindent\emph{Step 9: Conclusion.}
By Steps~5 and~8,
\[
A_1=\ov{\bigl(C_{\{1,2,3\}}\ov{A_1}\bigr)}\;C_{\{1,2,3\}}\in G .
\]
By Step~3, $s_2=(s_2A_1)\ov{A_1}\in G$, and by Lemma~\ref{lem:rotation}, $s_1=s_2^{\,\ov{R}}\in G$. Together with Step~1, the group $G$ contains $A_1$, $s_1$, $r_{1,2}$, $t_1$ and $u$. By Theorem~\ref{thm:wajnryb}, $G=\M{V_g}$.
\end{proof}

\begin{remark}
The relation (R9) is what makes Step~8 work. Steps~2--7 only produce elements in which the meridian twists occur in  combinations such as $A_i\ov{A_j}$ and $C_{i,j}\ov{A_k}$, or as squares, and the lantern relation turns these into the single twist $C_{\{1,2,3\}}$. Relation (R9) conjugates a twist about a curve enclosing two spots to a twist about a curve enclosing three spots. This lets us compare $C_{\{1,2,3\}}$ with a twist of the other kind and isolate $A_1$.
\end{remark}

\begin{proof}[Proof of Theorem~\ref{thm:threefortwo}]
 Let $G=\langle t_1,\,s_1,\, r_{1,2} A_2\rangle$ and set $F=  r_{1,2}A_2$. We show that $G$ contains the generators in Wajnryb's generators for the case $g=2$, which are $A_1$,   $s_1$, $r_{1,2}$ and $t_1$.

We use the same approach with Lemma~\ref{lem:split}. Using $\overline{t_1} A_2 = A_1\overline{t_1}$, $r_{1,2} A_2 = C_{1,2} r_{1,2}$  and $[r_{1,2}, A_1] = 1$, we have 
\[
    \overline{t_1} F t_1 F \overline{t_1} = \overline{t_1}  r_{1,2} A_2 t_1  r_{1,2} A_2 \overline{t_1} = C_{1,2} \overline{t_1} r_{1,2} t_1 r_{1,2} \overline{t_1} A_1 = C_{1,2} r_{1,2} A_1 \in G.
\]
 Then, multiplication of this element by $\overline{F}$ gives
\[
C_{1,2} r_{1,2} A_1 \overline{F} = C_{1,2} r_{1,2} A_1 \overline{A_2} \overline{r_{1,2}} = C_{1,2} A_1 r_{1,2} \overline{A_2} \overline{r_{1,2}} = C_{1,2} A_1 \overline{C_{1,2}} = C_{1,2} \overline{C_{1,2}} A_1 = A_1 \in G.
\]
Since $A_1^{t_1} = A_2 \in G$, $r_{1,2}A_2 \overline{A_2} = r_{1,2} \in G$. Thus, $G$ contains $A_1$, $r_{1,2}$, $s_1$, $t_1$, which are the Wajnryb's generators for the case $g=2$, so $G= \M{V_2}$.

\end{proof}

\section{Lower bounds}\label{sec:lower}

\begin{proposition}\label{prop:lower}
For $g\geq 2$, the group $\M{V_g}$ is not abelian. In particular, $d(\M{V_g})\geq 2$.
\end{proposition}

\begin{proof}
By (R2), $A_1^{t_1}=A_2$. Since $\alpha_1$ and $\alpha_2$ are not isotopic, $A_2\neq A_1$ \cite[Fact~3.6]{FarbMargalit}. Hence $t_1$ and $A_1$ do not commute. A group generated by one element is cyclic, hence abelian.
\end{proof}

Better lower bounds would follow from quotients of $\M{V_g}$ that need more generators. A natural source is the action on homology. Let $a_1,\dots,a_g,b_1,\dots,b_g$ be a symplectic basis of $H_1(\Sigma_g;\mathbb{Z})$ with $a_i=[\alpha_i]$. Every element of $\M{V_g}$ preserves the Lagrangian subspace $L=\ker\bigl(H_1(\Sigma_g;\mathbb{Z})\to H_1(V_g;\mathbb{Z})\bigr)$ spanned by $a_1,\dots,a_g$. Hence, in this basis, its action is given by a matrix
\[
\begin{pmatrix} P & Q\\ 0 & (P^{T})^{-1}\end{pmatrix},\qquad P\in \mathrm{GL}_g(\mathbb{Z}),\quad P^{-1}Q \text{ symmetric}.
\]
These matrices form the Siegel parabolic subgroup $P_g(\mathbb{Z})$ of $\mathrm{Sp}(2g,\mathbb{Z})$, and the induced homomorphism $\M{V_g}\to P_g(\mathbb{Z})$ is surjective \cite{Birman1975, Hirose2006}. Every lower bound for $d$ of a quotient of $P_g(\mathbb{Z})$, for instance of $\mathrm{GL}_g(\mathbb{Z})$ or of the reduction of $P_g(\mathbb{Z})$ modulo a prime, is a lower bound for $d(\M{V_g})$. We do not know a quotient of $\M{V_g}$ that needs more than two generators.

\end{document}